\documentclass[10pt]{article}
\usepackage{graphicx}
\usepackage[utf8]{inputenc}
\usepackage[top=2.5cm, bottom=2.6cm]{geometry}
\usepackage[T1]{fontenc}     
\usepackage{lmodern}
\usepackage{amsmath, amssymb, amsthm, latexsym, amsfonts, indentfirst, xcolor, mathtools, manfnt, microtype, subfig, cite,needspace}
\usepackage[breaklinks]{hyperref}
\usepackage{cleveref, enumerate, enumitem}
\usepackage{authblk}

\newcommand{\norm}[1]{\left\lVert#1\right\rVert}

\title{Unit distance graphs with few crossings per edge}

\author[1, 2]{Panna Geh\'er\thanks{Supported by ERC Advanced Grant `GeoScape' No.\ 882971 and by the National Research, Development and Innovation Office, NKFIH, ADVANCED 152590.  Email:~\href{mailto:geher.panna@ttk.elte.hu}{\tt geher.panna@ttk.elte.hu}.}}
\author[1, 2]{D\"{o}m\"{o}t\"{o}r P\'alvölgyi\thanks{Supported by the ERC Advanced Grant `ERMiD' No.\ 101054936 and by
the EXCELLENCE-24 project No.\ 151504 Combinatorics and Geometry of the NRDI Fund. Email:~\href{mailto:domotor.palvolgyi@ttk.elte.hu}{\tt domotor.palvolgyi@ttk.elte.hu}.}}
\author[1, 2]{D\'aniel G.\ Simon\thanks{Supported by ERC Advanced Grant `GeoScape' No.\ 882971. Email:~\href{mailto:dgsimon@renyi.hu}{\tt dgsimon@renyi.hu}.}}
\author[2, 3]{G\'eza T\'oth\thanks{Supported by ERC Advanced Grant `GeoScape' No.\ 882971 and by the National Research, Development and Innovation Office, NKFIH, K-131529 and ADVANCED 152590. Email:~\href{mailto:geza@renyi.hu}{\tt geza@renyi.hu}.}}

\affil[1]{
 ELTE E\"otv\"os Lor\'and University, Budapest, Hungary.
 }

\affil[2]{
HUN-REN Alfr\'ed R\'enyi Institute of Mathematics, Budapest, Hungary.
}

\affil[3]{
Department of Computer Science and Information Theory, Budapest University of Technology and Economics, Hungary.
}

\newtheorem{theorem}{Theorem}

\newtheorem{claim}{Claim}

\theoremstyle{definition}
\newtheorem{remark}{Remark}

\begin{document}
\maketitle

%%%%%%%%%%%%%%%%%%%%%%%%%%%%%%%%%%%%%%%%%
\begin{abstract}
A graph is called a $k$-planar unit distance graph if it can be drawn in the plane such that every edge is a unit line segment and is involved in at most $k$ crossings. We investigate $u_k(n)$, the maximum number of edges of such graphs on $n$ vertices. For $k=1$, we improve the best known upper bound, by showing that $u_1(n) \leq 3n - c\sqrt{n}$ for some constant $c>0$. This bound is tight up to the value of the constant $c$. For $k=2$, we establish the first non-trivial upper bound by proving that $u_2(n) \leq 4n - 8$. Regarding lower bounds we give a construction for $k=2$ that shows $u_2(n) \geq u_0(n) + c\sqrt{n}$ if $n$ is sufficiently large.
\end{abstract}

%%%%%%%%%%%%%%%%%%%%%%%%%%%%%%%%%%%%%%%%%

\section{Introduction}

A graph is called a \emph{matchstick graph} if it can be drawn in the plane with no crossings such that all of its edges are unit segments. This graph class was introduced by Harborth \cite{H81, H94}. He conjectured that the maximum number of edges of a matchstick graph on $n$ vertices is 
$3n - \left\lfloor\sqrt{12n - 3}\right\rfloor$ and he proved it \cite{H74} 
in the special case where the unit distance is the minimum distance among the points.
These graphs are known as \emph{penny graphs}. 
After almost 40 years, his conjecture was settled by Lavoll\'ee and Swanepoel \cite{LS22}. This result is tight, as shown by a hexagonal piece of the regular triangular lattice.

For any $k\ge 0$, a graph is called \emph{$k$-planar} if it can be drawn in the plane such that each edge is invol\-ved in at most $k$ crossings; such a drawing is called a $k$-plane drawing. Let $e_k(n)$ denote the maximum number of edges of a $k$-planar graph on $n$ vertices. For general $k$, we have $e_k(n)\le c\sqrt{k}n$ for some constant $c>0$, which is tight apart from the value of $c$ \cite{PT97, A19}. 
For small values of $k$, the exact values of the maximum number of edges are known. Since $0$-planar graphs are the well-known planar graphs, we have $e_0(n)=3n-6$ for $n\ge 3$, which is tight for triangulations. It is also known that $e_1(n)=4n-8$ for $n\ge 4$ (proved independently in \cite{BSW83} and \cite{PT97}) and that $e_2(n)\le 5n-10$ \cite{PT97}. These results are tight for infinitely many values of $n$.

Introduced in \cite{GT25}, a graph is called a \emph{$k$-planar unit distance graph} if it can be drawn in the plane such that each edge is a unit segment and is involved in at most $k$ crossings. Clearly,
$0$-planar unit distance graphs are exactly the matchstick graphs.

The maximum number of edges a $k$-planar unit distance graph on $n$ vertices can have is denoted by $u_k(n)$. Since $0$-planar unit distance graphs are the matchstick graphs, by the result of Lavoll\'ee and Swanepoel \cite{LS22}, 
we have $u_0(n)= 3n - \left\lfloor\sqrt{12n - 3}\right\rfloor$. For $k=1$, we have $u_1(n)\le e_1(n)=4n-8$. It was shown by the two extremal authors in \cite{GT25} that $u_1(n)\le 3n-\sqrt[4]{n}/15$. 
Here we improve this upper bound.
\begin{theorem} \label{thm_1planar}
For the maximum number of edges of a $1$-planar unit distance graph, $u_1(n)$, we have
    $$u_1(n) \leq 3n - \sqrt{n}/100.$$
\end{theorem}

This result is tight apart from the 
coefficient of the $\sqrt{n}$ term, 
as $u_1(n)\ge u_0(n)=3n - \left\lfloor\sqrt{12n - 3}\right\rfloor$. 
We did not attempt to find the optimal coefficient, 
probably $1/100$ is far from the optimal value.

Our proof is based on a recently introduced tool, called the Density Formula \cite{KKKRSU23}. 
For a given drawing, it describes the connection between the number of edges, vertices, crossings, and sizes of its cells. 
This result is particularly powerful in establishing upper bounds on the number of edges for 
many so-called beyond planar graph classes. 
Another important tool in our proof is the 
Isoperimetric inequality \cite{C13}. 
It states that a closed Jordan curve $\gamma$ of given length in the plane encloses the maximum area if it is a circle. 

For $k=2$, we have $u_2(n)\le e_2(n)=5n-10$. 
We establish the first improvement of this bound.

\begin{theorem} \label{thm_2planar}
For the maximum number of edges of a $2$-planar unit distance graph, $u_2(n)$, we have
    $$u_2(n) \leq 4n-8.$$
\end{theorem}

Our proof uses the technique developed in \cite{PT97}. 
We choose a maximal plane subgraph in a given 2-plane drawing, 
and estimate the number of the remaining edges using the discharging method.

Clearly, for every $n$, we have $u_2(n) \geq u_1(n)\geq u_0(n)$. 
As we have mentioned, the edge number $3n - \left\lfloor \sqrt{12n - 3}\right\rfloor$ 
can be achieved on $n$ vertices, even for matchstick graphs, 
by considering a hexagonal piece of the regular triangular lattice. 
However, it is not obvious that if we allow a few crossings per edge, 
we can improve this construction. Recently, \v{C}ervenkov\'a and Kratochv\'il 
\cite{CK25} proved that $u_1(n)> u_0(n)$ if $n$ is large enough. Moreover, they showed that for any constant $\alpha< \sqrt[4]{1/3}$, we have
$u_1(n)\ge u_0(n)+\alpha\sqrt[4]{n}$ if $n$ is sufficiently large. 
For $k=3$, an easy example (two copies of a proper piece of a triangular grid, shifted by a unit vector) gives $u_3(n)\geq 3.5n - c\sqrt{n}$, for some constant $c>0$. 
Here we give a construction for 2-planar unit distance graphs with roughly $u_0(n)+0.58 \sqrt{n}$ edges, if $n$ is large enough.
\begin{theorem} \label{thm_construction}
For all $n$ large enough, we have
    $$u_2(n) \geq u_0(n)+0.58\sqrt{n}.$$
    
\end{theorem}

\begin{remark}
According to our definition, 
a $k$-planar unit distance graph admits a drawing satisfying both conditions simultaneously: each edge is a unit segment and  involved in at most $k$ crossings. 
Note that this is not the same class as graphs that are $k$-planar and unit distance graphs; 
for instance, the Moser spindle is a planar graph and also a unit distance graph. 
However, it does not admit a crossing-free drawing such that each edge is a unit segment, and thus it is not a matchstick graph.
\end{remark}

%%%%%%%%%%%%%%%%%%%%%%%%%%%%%%%%%%%%%%%%%%%%%%%%%%%%%%%%%
\section{1-planar unit distance graphs}

\begin{proof}[\bf Proof of \Cref{thm_1planar}] 
In this section, we prove \Cref{thm_1planar}. 
Our proof is based on 
the Density Formula \cite{KKKRSU23}. First we need some preparation.
Fix a drawing $\Gamma$ of some graph $G$ and let $\mathcal{X}$ denote 
the set of all crossings in $\Gamma$. 
The planarization of the drawing $\Gamma$ is the plane drawing obtained from $\Gamma$ 
by adding new vertices at each crossing point and splitting the involved 
edges into edge-segments. We say that the
drawing $\Gamma$ is connected if the graph underlying its planarization is connected.

Furthermore, let $\mathcal{C}$ denote the set of cells i.e., the connected components of $\mathbb{R}^2$ remaining after removing all edges and vertices in $\Gamma$.  We say that a crossing $\alpha \in \mathcal{X}$ and a cell $F$ are incident, if $\alpha$ is on the boundary of $F$. The total number of edge-segment-incidences and vertex-incidences of a cell $c \in \mathcal{C}$, both counted with multiplicity, is called the \emph{size} of $c$ and is denoted by
$\norm{c}$. Note that the incident crossings are not counted in
$\norm{c}$. 
The Density formula gives an upper bound on the number of edges in terms of its number of vertices, crossings, and sizes of its cells.

\begin{theorem} [Density formula {\cite[Lemma 3.1]{KKKRSU23}}] \label{thm_density}
Let $t$ be a real number. Let $\Gamma$ be a connected drawing of a graph $G = (V,\,  E)$ with at least one edge. Then for $|E|$,
the number of edges in $G$, we have
\begin{align}
|E| \leq t(|V|-2) - \sum_{c \in \mathcal{C} }\left( \frac{t-1}{4} \norm{c}-t \right) -|\mathcal{X}|.
\end{align}
\end{theorem}

We are now ready to prove \Cref{thm_1planar}.
Let $G$ be a $1$-planar unit distance graph with $n$ vertices and consider a $1$-plane unit distance drawing $\Gamma$. Clearly, we can assume that $\Gamma$ is a connected drawing.
Let $\mathcal{X}$ denote the set of all crossings in $\Gamma$ and let $\mathcal{C}$ denote the set of cells.

Remove an edge of $G$ from each crossing. We get a matchstick graph 
with $|E|-|\mathcal{X}|$ edges. By the result of Lavoll\'ee and Swanepoel \cite{LS22},  
$|E|-|\mathcal{X}|\le 3n - \left\lfloor\sqrt{12n - 3}\right\rfloor<3n-\sqrt{12n}$, so
$|E|<3n-\sqrt{12n}+|\mathcal{X}|$. Therefore, we are done if $|\mathcal{X}|<c\sqrt{n}$. 
Assume for the rest of the proof that $|\mathcal{X}|\ge c\sqrt{n}$. 

By setting $t=3$ in \Cref{thm_density}, we have
\begin{align*}
|E| &\leq 3n-6 - \sum_{c \in \mathcal{C} }\left( \frac{1}{2} \norm{c}-3 \right) -|\mathcal{X}|.
\end{align*}

Note that the quantity $\frac{1}{2}\norm{c}-3$ is non-negative for any cell $c$ of size at least $6$. Thus, to prove the upper bound on $|E|$, we only need to consider cells with size at most $5$. 
Such a cell can contain (i) three edge segments and at most 
two vertices, (ii) four edge segments and at most
one vertex, and (iii) five edge segments and no vertices. 
See \Cref{fig_5face} for all six possibilities. However, apart from 
case (a), all possibilities contain an edge that is crossed by 
two other edges, contradicting 1-planarity. So, all cells of $\Gamma$ with size at most $5$ belong to case (a).

\begin{figure}[h!]
    \centering
    \includegraphics[width=0.98\linewidth]{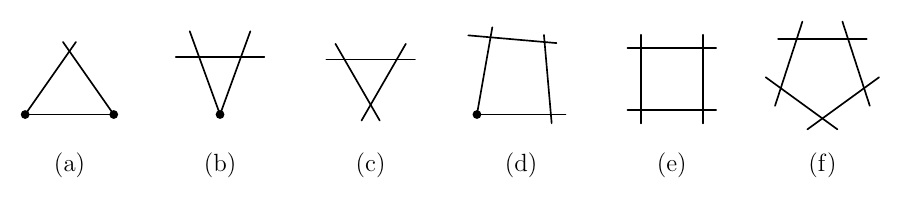}
    \caption{All types of cells with size at most $5$ 
    that can occur in the planarization of a straight line drawing of a 
    graph.}
    \label{fig_5face}
\end{figure}

By the previous observations, we have
\begin{align}
|E| \leq 3n + \frac{c_5}{2} -|\mathcal{X}|,
\end{align}
where $c_5$ denotes the number of cells of size exactly 5.

If $|\mathcal{X}|> \frac{c_5}{2} + \sqrt{n}/100$, 
then $|E| \leq 3n + \frac{c_5}{2} -|\mathcal{X}|
    \leq 3n - \sqrt{n}/100$ and we are done. 
    Therefore, for the rest of the proof we can assume that 

\begin{align} \label{upper_bound}
    |\mathcal{X}| \leq \frac{c_5}{2} + \sqrt{n}/100.
\end{align}

First we describe some simple properties of cells of size $5$.

\needspace{3\baselineskip}
\begin{claim} \label{claim_size5} $ $
\begin{enumerate}[label=(\roman*)] 
    \item\label{claim_size5_item1}
    If a crossing is incident to two cells of size 5, then they share a common edge-segment.
    \item\label{claim_size5_item2} A crossing can be incident to at most two cells of size $5$.
\end{enumerate}
\end{claim}

    \begin{figure}[ht!]
    \centering
    \includegraphics[width=0.3\linewidth]{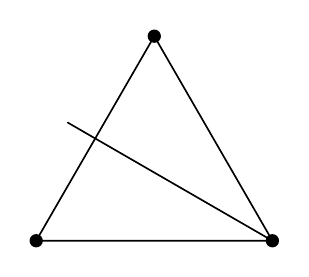}
    \caption{The unique configuration of two cells of size $5$ with a common crossing.}
    \label{fig_5facecommoncrossing}
\end{figure}

\begin{proof}
(i) Let $\alpha$ be a crossing defined by the edges $u_1v_1$ and $u_2v_2$, 
and suppose that $\alpha$ is incident to two cells of size $5$ that do not share a common edge-segment. 
Assume without loss of generality that the uncrossed boundary edges of these cells are $u_1u_2$ and $v_1v_2$, respectively. This means that the vertices $u_1, \, u_2, \, v_2, \, v_1$ form a rhombus with unit edges, such that  $u_1v_1$ and $u_2v_2$ are opposite edges. Therefore, they cannot cross, a contradiction.

(ii) First, note that a crossing can be incident to at most 4 cells. Suppose for contradiction that a crossing is incident with at least three cells of size $5$. 
Then, among these  cells there must be a pair that cannot share a common edge, which contradicts \Cref{claim_size5}~\ref{claim_size5_item1}.
\end{proof}

So the only possible configuration a crossing can be incident to two cells of size $5$ is the one depicted in \Cref{fig_5facecommoncrossing}.
By \Cref{claim_size5}~\ref{claim_size5_item2}, a crossing can be incident with at most two cells of size $5$. 
On the other hand, each cell of size $5$ is incident to a crossing (see \Cref{fig_5face}).
Therefore, $|\mathcal{X}| \geq \frac{c_5}{2}$.
Combining it with \eqref{upper_bound}, we have

\begin{align} \label{combined}
\frac{c_5}{2}\le |\mathcal{X}| \le \frac{c_5}{2}+\sqrt{n}/100.
\end{align}

\begin{claim}\label{claim_equilateral}
    There are at least $c_5-|\mathcal{X}|$ pairwise interior-disjoint unit equilateral triangles in the drawing $\Gamma$.
\end{claim}
\begin{proof}
    Let $x_1, \, x_2$ be the number of crossings that are incident with exactly one and two cells of size~$5$, respectively. Clearly, $|\mathcal{X}|\geq x_1+x_2$, and by \Cref{claim_size5}~\ref{claim_size5_item2}, we have $x_1+2x_2\geq c_5$. Therefore, there are $x_2\geq c_5-|\mathcal{X}|$ crossings incident with two cells of size $5$. By \Cref{claim_size5} 
    the only configuration where two cells of size $5$ are incident to a common crossing is the one shown in \Cref{fig_5facecommoncrossing}.

    This configuration contains a unit equilateral triangle. 
    Let $\mathcal{T}$ be the set of equilateral triangles 
    determined by the pairs of cells with size 5 which share a common crossing. 
    Clearly, the triangles in $\mathcal{T}$ are pairwise interior-disjoint, and $|\mathcal{T}|=x_2$. 
    Consequently, $\Gamma$ contains at least $x_2 \geq c_5- |\mathcal{X}|$ 
    pairwise interior-disjoint equilateral triangles.
\end{proof}

Consider a plane subdrawing $\Gamma_0$ of $\Gamma$ with maximum number of edges, and among those the one with the minimum number of triangular faces.
This last property is rather technical, we only need it when applying \Cref{thm_GT1}. We call the underlying graph $G_0$. We denote the edge set of $G$ by $E$, the edge set of $G_0$ by $E_0$ and the set of remaining edges, $E\setminus E_0$, by $E_1$. Let $e_0=|E_0|$ and $e_1=|E_1|$.
Due to the maximality of $\Gamma_0$, every edge $e \in E_1$ crosses at least one edge in $E_0$.

Let $f$ be the number of faces of $\Gamma_0$, including the unbounded face, and let $\Phi_1, \, \Phi_2 \, \dots, \, \Phi_f$ be the faces of $\Gamma_0$.

For $i \ge 3$, let $f_i$ denote the number of faces $\Phi$ of $\Gamma_0$ that are bounded 
by $i$ edges. 

Let $F_{\ge 5} = \sum_{i=5}^{\infty} i f_i$, that is, the total number of edges of the faces of size at least $5$. The following result from \cite{GT25} will be useful in our proof.

\begin{theorem}[\!\!{\cite[Claims 3 and 4]{GT25}}] \label{thm_GT1}  Let $p>0$ integer. 
\begin{enumerate}[label=(\roman*)] 
    \item\label{item1}  If $F_{\geq5}\geq p$,  
    then $e=e_0+e_1\leq3n-p/10$.

\item\label{item2} If $f_3\geq p$, 
then $e=e_0+e_1\leq3n-\sqrt{p}/5$.
\end{enumerate}
\end{theorem}

We will also use the Isoperimetric inequality.

\begin{theorem} [Isoperimetric inequality \cite{C13}] \label{thm_isoperimetric}
Let $P$ be a polygon with perimeter $\ell(P)$ and area $A(P)$. Then, we have $\ell^2(P)\ge 4\pi A(P)$.
\end{theorem}

If $|E| < 3n - \sqrt{n}/100$, we are done.  Therefore, we can assume for the rest of the proof that $|E| \ge 3n - \sqrt{n}/100$.
We distinguish two cases.
\medskip

\textbf{Case $1$.} $|\mathcal{X}|\ge n/10$. Then by \eqref{combined} 
and by \Cref{claim_equilateral}, we obtain that 
the drawing $\Gamma$ contains at least 
$c_5-|\mathcal{X}|\geq 
|\mathcal{X}|-\sqrt{n}/50\geq n/10-\sqrt{n}/50>n/20$
disjoint unit equilateral triangles.

Define the area and perimeter, $A(\Gamma)$ and $\ell(\Gamma)$, respectively, 
as the area and perimeter of the (possibly degenerate) polygon 
surrounded by the unbounded face of $\Gamma$. Each unit equilateral 
triangle has area $\frac{\sqrt3}{4}$, thus for the area of $\Gamma$, we have 
$$A(\Gamma)\geq \frac{n}{20}\frac{\sqrt{3}}{4}.$$
By \Cref{thm_isoperimetric}, we have
$$\ell^2(\Gamma)\geq4\pi A(\Gamma)\geq  \pi\sqrt{3}\frac{n}{20}>\frac{n}{4}$$   
consequently $\ell(\Gamma)\ge \sqrt{n}/2$.

Let $\mathcal{B}$ denote the edges that lie on the boundary of $\Gamma$ 
or contain an edge-segment that is on the boundary of $\Gamma$.
Each edge-segment has length at most one, therefore there are 
at least $\ell(\Gamma)\ge \sqrt{n}/2$ edges in $\mathcal{B}$. 
These edges can only be crossed by other edges of $\mathcal{B}$. 
This means that at least half of the edges in $\mathcal{B}$ 
also lie on the boundary of $\Gamma_0$. 
Thus, $\ell(\Gamma_0)\geq \frac12\ell(\Gamma)\ge \sqrt{n}/4$. 
By definition, $F_{\geq5}\geq \ell(\Gamma_0)\ge \sqrt{n}/4$. 
Applying \Cref{thm_GT1}~\ref{item1}, 
we obtain that 
$$e\leq 3n-F_{\geq5}/10\leq 3n-\sqrt{n}/40$$
concluding the proof in this case.

\medskip

\textbf{Case $2$.} $|\mathcal{X}|\le n/10$.
Remove an edge of $G$ from each crossing. 
We get a matchstick graph with $|E|-|\mathcal{X}|$ edges. 
By the maximality of $\Gamma_0$, we have $e_0\ge |E|-|\mathcal{X}|\ge 3n-\sqrt{n}/100-n/10=(3-1/10)n-\sqrt{n}/100>(3-1/9)n$.
Add at most $n/9-6$ extra (not necessarily unit or straight-segment) 
edges to $G_0$ to get a triangulation, which has 
$3n-6$ edges and $2n-4$ triangular faces. Each extra edge can create at most $2$ triangles. 
Hence, $\Gamma_0$ contains at least
$2n-4-2n/9+12>16n/9$
triangular faces. 
By applying \Cref{thm_GT1}~\ref{item2}, 
we obtain that
$e\le 3n-\sqrt{f_3}/5\le 3n-\sqrt{16n/9}/5=3n-4\sqrt{n}/15<3n-\sqrt{n}/100$
which completes the proof in the second case and hence the proof of \Cref{thm_1planar}.
\end{proof}

%%%%%%%%%%%%%%%%%%%%%%%%%%%%%%%%%%%
\section{2-planar unit distance graphs}

The aim of this section is to prove \Cref{thm_2planar} and \Cref{thm_construction}. We first establish the upper bound.

\begin{proof}[\bf Proof of \Cref{thm_2planar}.]
Let $G$ be a $2$-planar unit distance graph with $n$ vertices and consider a $2$-plane unit distance drawing of $G$. Let $G_0$ be a plane subgraph of $G$ with maximum number of edges, and among those one with the minimum number of triangular faces. We denote the edge set of $G$ by $E$, the edge set of $G_0$ by $E_0$ and the set of remaining edges, $E\setminus E_0$, by $E_1$.
Let $f$ be the number of faces of $G_0$, including the unbounded face, and let $\Phi_1, \, \Phi_2, \, \dots, \, \Phi_f$ be the faces of $G_0$. For any face $\Phi_i$, $|\Phi_i|$ is the number of its bounding edges, counted with multiplicity. 
Due to the maximality of $G_0$, every edge $e \in E_1$ crosses at least one edge in $E_0$.
Using the definition introduced in \cite{PT97}, we call
the closed portion between an endpoint of $e \in E_1$ and the nearest crossing of $e$ with an edge of $E_0$ a \emph{halfedge}. 
Every edge of $E_1$ contains two halfedges. A halfedge lies in a face $\Phi_i$ of $G_0$, it intersects an edge of $\Phi_i$ and crosses at most one other halfedge of $\Phi_i$. We denote the edge in $E_1$ containing the halfedge $\alpha$ by $\widehat{\alpha}$.

We use the discharging method to give an upper bound 
on the number of halfedges. First, we assign charges to the faces of $G_0$. Initially, for each face $\Phi_i$, define its charge 
$c(\Phi_i)$ as
$$c(\Phi_i)= 2 \cdot t(\Phi_i) + |\Phi_i| -2-h(\Phi_i),$$
where $t(\Phi_i)$ denotes the number of 
edges needed to triangulate the face $\Phi_i$ and $h(\Phi_i)$ denotes the number of halfedges lying in the face $\Phi_i$.
Our goal is to show that we can distribute the charges such that all faces have nonnegative charge.

First, we show  that for the initial charges $c(\Phi_i)$, 
this holds for all faces of at least $4$ vertices (including triangles whose 
boundaries are not connected, that is, triangles containing some isolated vertices).

\begin{claim} \label{claim_halfedges}
Let $\Phi$ be a face of $G_0$ that is not a triangle with a connected boundary. Then
$$h(\Phi) \leq 2 \cdot t(\Phi) + |\Phi| -2.$$
\end{claim}

\begin{proof}
As the edges of $G$ can be involved in at most $2$ crossings, we have $h(\Phi) \leq 2 |\Phi|$. A straightforward consequence of Euler's formula is the following statement. If the boundary of $\Phi$ has $m$ connected components, then
\begin{equation} \label{Euler_consequence}
t(\Phi)=|\Phi|+3m-6. 
\end{equation}

Suppose first that the boundary of $\Phi$ is not connected, that is, $m \geq 2$. We have

$$h(\Phi) \leq 2 |\Phi| \leq 2\cdot t(\Phi) \leq 2 \cdot t(\Phi) + |\Phi| -2,$$
and we are done in this case.

\smallskip

From now on, assume that the boundary of $\Phi$ is connected, that is, $m=1$. 
In this case, we have to show that $h(\Phi) \leq 3 |\Phi|-8$. First, let $|\Phi| \geq 6$. 
Again by 2-planarity, we have $h(\Phi) \leq 2 |\Phi|$. Moreover, we claim that 
\begin{align} \label{minimalhalfedges}
    h(\Phi) \leq 2 |\Phi| -2.
\end{align}

For that, consider a maximal set of halfedges 
$\mathcal{H}= \{ \alpha_1, \, \alpha_2, \,  \dots, \alpha_k \}$ 
in $\Phi$ such that they do not cross each other 
and the edges of $\Phi$ are only crossed at most once. 
We can assume that we have at least two halfedges, otherwise the statement is trivial. 
A halfedge $\alpha_i \in \mathcal{H}$ divides $\Phi$ into two parts. 
Since the halfedges in $\mathcal{H}$ do not cross each other, 
all other halfedges $\alpha_j \in \mathcal{H} \setminus \{\alpha_i\}$ are entirely in one of these two parts.  
Moreover, if one part does not contain any halfedges, then $\alpha_i$ is called a {\em minimal halfedge}. 
There are at least 
two minimal halfedges among the elements of $\mathcal{H}$, say $\alpha_1$ and~$\alpha_2$. 
Note that for a minimal halfedge $\alpha_i$, there exists at least one corresponding edge 
$e(\alpha_i)$ of $\Phi$ that is not crossed by the elements of $\mathcal{H}$. 
Moreover, any additional halfedge, not in $\mathcal{H}$, 
that crosses $e(\alpha_i)$, also crosses the halfedge $\alpha_i$. 
Therefore, the edges $e(\alpha_1)$ and $e(\alpha_2)$ 
can be involved in at most one crossing, so $\Phi$ can have at most $2|\Phi| -2$ halfedges.
(A very similar argument was used in \cite{GT25}.)

For $|\Phi| \geq 6$ and by \eqref{minimalhalfedges}, we have
$h(\Phi) \leq 2 |\Phi| -2 \leq 3|\Phi|-8$, and we are done in this case.
We check the cases $|\Phi| = 4$ and $|\Phi| =5$ separately. 
\begin{itemize}
    \item Suppose that $|\Phi| = 4$.
     We need to show that $h(\Phi) \leq 4$. Let the vertices of 
     $\Phi$ be $v_1, \, v_2, \, v_3, \, v_4$, in this order. 
     If none of the edges of $\Phi$ is intersected by two halfedges, 
     then there are at most $4$ halfedges in $\Phi$ and we are done. 
     So we can assume, without loss of generality, that the edge $v_1v_4$ is intersected by two halfedges, $\alpha$ and $\beta$. 
     Observe that $\alpha$ and $\beta$ cannot have different endvertices 
     in $\Phi$, as~$\Phi$ is a rhombus, thus only one of $v_2$ and $v_3$ 
     can be at distance less than one from the edge~$v_1v_4$.
     Therefore, we can further suppose that $v_3$ is an endvertex 
     of both $\alpha$ and $\beta$, see \Cref{fig_rhombus}.
     Clearly, the edge $v_3v_4$ cannot be involved in any crossing, 
     thus it is enough to show that the number of the halfedges 
     in $\Phi$ intersecting the edge $v_1v_2$ and the edge $v_2v_3$ 
     is at most two altogether. 
     Indeed, any halfedge intersecting $v_1v_2$ crosses any halfedge 
     intersecting $v_2v_3$, thus besides $\alpha$ and $\beta$, there are only at most two more halfedges.

\begin{figure}[ht]
    \centering
    \includegraphics[width=0.55\linewidth]{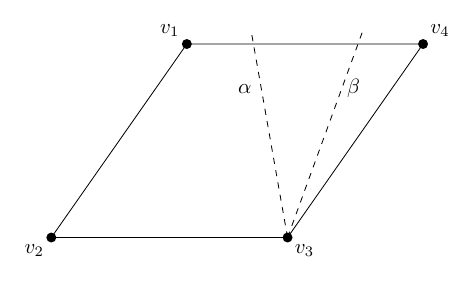}
    \caption{Let $\Phi$ be a rhombus with vertices  $v_1, \, v_2, \, v_3, \, v_4$, as in the figure. If there are two halfedges $\alpha, \, \beta$ intersecting the edge $v_1v_4$, then they share an endpoint in $\Phi$ (which is either $v_2$ or $v_3$).}
    \label{fig_rhombus}
\end{figure}
     
    \item Suppose  that $|\Phi| = 5$.
    We need to show that $h(\Phi) \leq 7$.

    \textbf{Case 1.}
    First, assume that there is an edge that appears twice in the boundary 
    of $\Phi$, and $\Phi$ has only four vertices. 
    Let the vertices of $\Phi$ be $v_1, \, v_2, \, v_3, \, v_4$. 
    The boundary of $\Phi$ contains a unit equilateral triangle $T$, 
    with vertices, say $v_1, \, v_2, \, v_3$. 
    Note that $v_4$ cannot be in the interior of $T$, as none of the 
    interior points of $T$ can determine unit distance with the vertices 
    $v_1, \, v_2, v_3$. Thus $v_4$ lies outside of $T$ and $\Phi$ is the unbounded face of $G_0$. 
    As the boundary of $\Phi$ is connected, the vertices of $G_0$, 
    besides $v_1, \, v_2, \, v_3, \, v_4$, are isolated vertices inside $T$. 
    Therefore, all halfedges of $\Phi$, intersect the boundary of $T$.
    Thus, $\Phi$ contains at most $6$ halfedges.

        \textbf{Case 2.}
    Now assume that all edges of the boundary of $\Phi$ appear only once, 
    and $\Phi$ has five vertices. A
    halfedge $\alpha$ in $\Phi$ divides $\Phi$ into two parts. 
    Let $a(\alpha)$ and $b(\alpha)$ be the number of vertices of $\Phi$ in the two parts. 
    For any halfedge $\alpha$ in $\Phi$, we either have $a(\alpha)=2$, $b(\alpha)=4$ 
    (or vice versa), or $a(\alpha)=3$, $b(\alpha)=3$. 
    In the first case, we say that $\alpha$ is a halfedge of type 1, 
    otherwise it is a halfedge of type 2. 
    Observe that if $\alpha$ is of type 1, then there is a corresponding edge $e(\alpha)$ 
    with the following property: each halfedge that crosses $e(\alpha)$, 
    also crosses the halfedge $\alpha$. Therefore, there is at most one such halfedge. 
    Moreover, for any edge $e$ of $\Phi$, there are at most two halfedges $\alpha_1, \, \alpha_2$ 
    of $\Phi$ with $e(\alpha_1)= e(\alpha_2)$, and then this edge is not crossed by any halfedge. 
    So we can conclude that each halfedge prevents an intersection on a boundary edge of $\Phi$.
    Thus, if $\Phi$ contains at least $3$ halfedges of type 1, then $3$ out of the possible intersections are prevented, and we are done. 
    Assume that $\Phi$ contains at most $2$ halfedges of type 1. 
    Observe that any two halfedges of type 2 with different end vertices must cross each other. 
    For any vertex $v$ of $\Phi$, there are at most two halfedges of type 2 emanating from $v$. 
    Thus, by the previous observation, $\Phi$ can have at most two halfedges of type 2, and we are done. \qedhere
\end{itemize}
\end{proof}

Now, we examine triangles with connected boundaries.
It is easy to see that they can contain at most two halfedges (for details, see \cite[Lemma 2.1]{PT97}). We call a triangle with a connected boundary, containing two halfedges, a \emph{bad triangle}. Let $\Phi$ be a bad triangular face of $G_0$ with halfedges $\alpha$ and~$\beta$. We distinguish four types of bad triangles, see \Cref{fig_triangles}. 

\begin{itemize}
    \item Type 1:
    Halfedges $\alpha$ and $\beta$ do not share a common endvertex in $\Phi$.
\end{itemize}

In the remaining cases, halfedges $\alpha$ and $\beta$ do share a common endvertex in $\Phi$, and hence they cross the same edge of $\Phi$, say the edge $e$.

\begin{itemize}
    \item Type 2:
    At most one of $\widehat{\alpha}$ and $\widehat{\beta}$ is crossed by an edge in $E_0 \setminus \{ e\}$. 

    \item Type 3:
   Edges $\widehat{\alpha}$ and $\widehat{\beta}$ are both crossed by the same edge $f \in E_0 \setminus \{ e\}$. 

    \item Type 4:
Edges $\widehat{\alpha}$ and $\widehat{\beta}$ are crossed by different edges $f_1, \, f_2 \in E_0 \setminus \{ e\}$. 
\end{itemize}

\begin{figure}[ht!]
    \centering
    \includegraphics[width=0.95\linewidth]{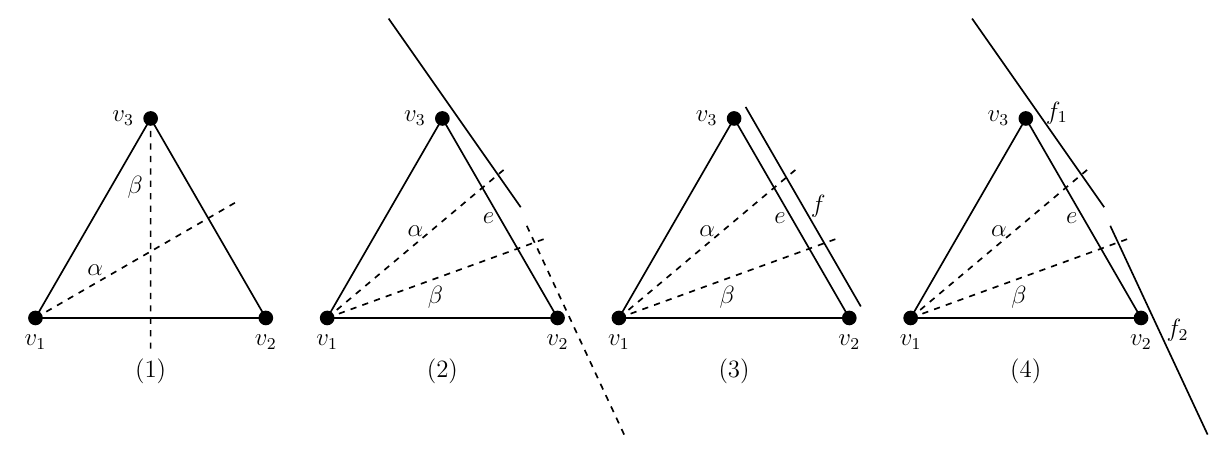}
    \caption{The four types of bad triangles. Edges in $E_1$ are drawn with dashed lines.}
    \label{fig_triangles}
\end{figure}

We show that due to the choice of $G_0$, bad triangles of types 1, 2 and 3 do not exist. 

\begin{claim}
All bad triangles are of type 4.
\end{claim}

\begin{proof}
Let $\Phi$ be a bad triangle with vertices $v_1, \, v_2,\, v_3$, in this order and with halfedges $\alpha$ and $\beta$.
We claim that in case $\Phi$ is of type 1, 2 or 3, then the number of triangles in $G_0$ can be reduced by `edge flips’, i.e., by replacing some edges of $E_0$ by edges of $E_1$.
More precisely, we make the following changes depending on the type of $\Phi$.

\begin{itemize}
    \item Let $\Phi$ be of type 1. Suppose that $\alpha$ emanates from $v_1$. Replace the edge $v_2v_3$ by $\widehat{\alpha}$ in $G_0$. 
    
    \item Let $\Phi$ be of type 2. Suppose that both $\alpha$ and $\beta$ emanate from $v_1$ and also that $\widehat{\beta}$ is not crossed by any edge in $E_0 \setminus \{e\}$.
    Replace edge $e$ by $\widehat{\beta}$.
    
    \item Let $\Phi$ be of type 3. Suppose that both $\alpha$ and $\beta$ emanate from $v_1$ and both $\widehat{\alpha}$ and $\widehat{\beta}$ crosses $e$ and $f \in E_0$. Replace the edges $e$ and $f$ by $\widehat{\alpha}$ and $\widehat{\beta}$ in $G_0$. 
\end{itemize}

In all three cases, after executing the edge flips, 
we obtain another plane subgraph of $G$ with the same number of edges.
The triangular face $\Phi$ has disappeared and it is easy to see that we did not create any new triangular faces. 
Hence, we reduced the number of triangles in $G_0$, 
which is a contradiction. We can conclude that all bad triangles are of type 4.
\end{proof}

In order to get nonnegative charges for the triangular faces, we discharge the weights $c(\Phi_i)= 2 \cdot t(\Phi_i) + |\Phi_i| -2-h(\Phi_i)$ as follows. 
Let $\Phi$ be a bad triangle with halfedges $\alpha$ and $\beta$. Assume that they both intersect the edge $e$ of $\Phi$, and also $\widehat{\alpha}$ crosses $f_1 \in E_0$, $\widehat{\beta}$ crosses $f_2 \in E_0$. 
Let $\Psi$ be the neighboring face of $\Phi$ containing the edge $e$. Note that $\Psi$ also contains the edges $f_1$ and $f_2$.
Increase the charge of $\Phi$ by $1$ and decrease the charge of $\Psi$ by $1$. 
We say that $\Psi$ is the helper of $\Phi$. We first observe that $\Psi$ is bounded by at least $5$ edges.

\begin{claim} \label{claim_helper}
Let $\Phi$ be a bad triangle and let $\Psi$ be its helper. Then
$|\Psi| \geq 5.$
\end{claim}
\begin{proof}
Let $\Phi$ be a bad triangle with halfedges $\alpha$ and $\beta$.
Let the vertices of $\Phi$ be $v_1, \, v_2, \, v_3$.
First, we show that $\Psi$ cannot be a triangle. Suppose to the contrary and let the vertices of $\Psi$ be $v_1, \, v_3, \, v_4$. Then the union of $\Phi$ and $\Psi$ is a rhombus with vertices $v_1,\, v_2, \, v_3, \, v_4$, in this order. Halfedges $\alpha$ and $\beta$ are emanating from vertex $v_2$, and we can suppose that $\widehat{\alpha}$ intersects the edge $v_4v_1$, while $\widehat{\beta}$ intersects the edge $v_3v_4$, see \Cref{fig_helper}, (a). Let $P$ be the intersection of $\widehat{\alpha}$ and the edge $v_4v_1$, marked with a circle in the figure. 
The triangle $v_1v_2P$ is an obtuse triangle, the angle $\angle v_2v_1P=120^{\circ}$. Hence, the length of $v_2P$ is strictly larger than the length of edge $v_1v_2$. Thus, $v_2P$, and also the edge $\widehat{\alpha}$ have length more than one, which is a contradiction.

\begin{figure}[ht!]
    \centering
    \includegraphics[width=1\linewidth]{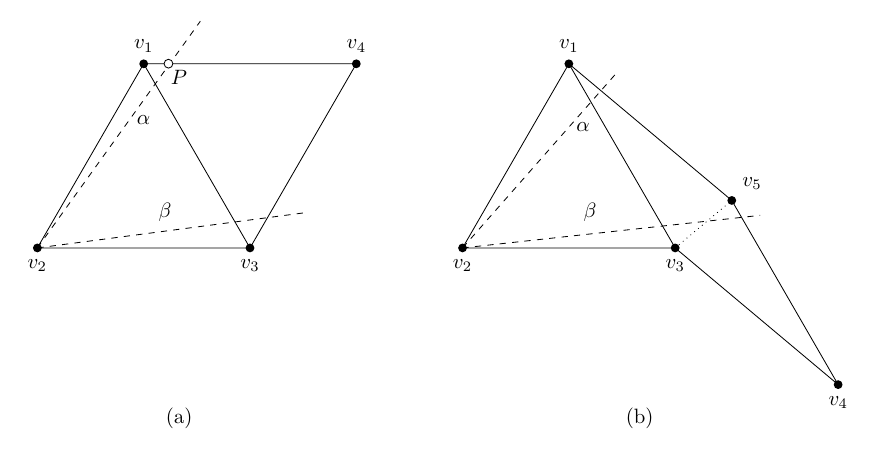}
    \caption{$\Phi$ is a bad triangle with halfedges $\alpha$ and $\beta$. Let $\Psi$ be the helper of $\Phi$. Then (a) $\Psi$ cannot be a triangle, and (b) $\Psi$ cannot be a quadrilateral.}
    \label{fig_helper}
\end{figure}

Next, we show that $\Psi$ cannot be a quadrilateral. Suppose to the contrary and let the vertices of $\Psi$ be $v_1, \, v_3, \, v_4, \, v_5$, in this order. Halfedges $\alpha$ and $\beta$ are emanating from vertex $v_2$, and we can suppose that $\widehat{\alpha}$ intersects the edge $v_5v_1$. If the angle $\angle v_5v_1v_3$ is at least $60^{\circ}$, then, by the previous argument, the length of $\widehat{\alpha}$ would be more than one, a contradiction. So, we can assume that the angle $\angle v_5v_1v_3$ is smaller than $60^{\circ}$. Clearly, $\widehat{\beta}$ cannot cross the edge $v_3v_4$, thus it crosses $v_4v_5$, see \Cref{fig_helper}, (b). Observe that $\widehat{\beta}$ must cross the line segment $v_3v_5$ (that is not an edge of $G$, drawn with a dotted line segment in the figure).
Since the angle $\angle v_5v_1v_3$ is smaller than $60^{\circ}$, angle $\angle v_1v_3v_5$ is larger than $60^{\circ}$, and again, by the previous argument, the length of $\widehat{\beta}$ is more than one, which is a contradiction.
We can conclude that $\Psi$ is bounded by at least 5 edges.
\end{proof}

Repeat this procedure for all bad triangles. We denote the final charge of a face $\Phi_i$ by $c'(\Phi_i)$.
Next, we prove that the modified charges are nonnegative. 

\begin{claim} \label{main_lemma}
For each face $\Phi_i$ in $G_0$, we have
    $c'(\Phi_i) \geq 0$.
\end{claim}
\begin{proof}
Let $\Phi_i$ be an arbitrary face of $G_0$.
By definition, if $|\Phi_i|=3$, we have $c'(\Phi_i) \geq 0$, and we are done in this case.
If $|\Phi_i|=4$, then by \Cref{claim_helper}, $c'(\Phi_i)=c(\Phi_i)$, so we are done 
by \Cref{claim_halfedges}.
Now let us assume that $|\Phi_i| \geq 5$.
If $c'(\Phi_i) = c(\Phi_i)$, we are also done by \Cref{claim_halfedges}, 
so we can assume that $c'(\Phi_i) < c(\Phi_i)$. 
Let $\Phi_i$ be the helper of a bad triangle with halfedges $\alpha$ and $\beta$.
First, note that other than $\alpha$ and $\beta$, edges 
$\widehat{\alpha}$ and $\widehat{\beta}$ cannot contain halfedges 
lying in a bad triangle, otherwise edges $\widehat{\alpha}$ 
and $\widehat{\beta}$ would have length at least $\sqrt{3}$.
Also, by definition, $\widehat{\alpha}$ and $\widehat{\beta}$ 
both cross the boundary of $\Phi_i$ twice.  
If the boundary of $\Phi_i$ contains all of its edges only once, 
then $\widehat{\alpha}$ and $\widehat{\beta}$ do not contain 
a halfedge that is lying in $\Phi_i$. On the other hand, if e.g.\ $\widehat{\alpha}$ 
crosses an edge of $\Phi_i$ that appears on the boundary twice, 
then it contains one halfedge, which lies in $\Phi_i$.

Therefore, the total number of crossings on the edges of $\Phi_i$ 
is at least $h(\Phi_i) + 4(c(\Phi_i)- c'(\Phi_i)) - 2(|\Phi_i|-\ell(\Phi_i))$, 
where $\ell(\Phi_i)$ denotes the number of different edges bounding $\Phi_i$, i.e., we count each edge of $\Phi_i$ only once (thus $\ell(\Phi_i) \leq |\Phi_i|$).

As the edges of $G$ can be involved in at most $2$ crossings, we have
\begin{align*}
    4(c(\Phi_i)- c'(\Phi_i)) + h(\Phi_i) - 2(|\Phi_i|-\ell(\Phi_i))&\leq 2 \cdot \ell(\Phi_i)\\
    4(c(\Phi_i)- c'(\Phi_i)) + h(\Phi_i) &\leq 2 |\Phi_i|\\
    c(\Phi_i)- c'(\Phi_i) + h(\Phi_i) &\leq 2 |\Phi_i|-3\\
    2 \cdot t(\Phi_i)+ |\Phi_i| -2 - c'(\Phi_i) &\leq 2 |\Phi_i|-3\\
     c'(\Phi_i) &\geq  - |\Phi_i| +2 \cdot t(\Phi_i) +1 \geq |\Phi_i| -5 \geq 0,
\end{align*}
where in the last line we used the fact that $t(\Phi_i) \geq |\Phi_i|-3$ for all faces $\Phi_i$. This concludes the proof of \Cref{main_lemma}.
\end{proof}

%%%%%%%%%%%%%%%%%%%%
We are ready to prove \Cref{thm_2planar}.

For the number of edges in $G_0$, we have

\begin{equation} \label{eq_edges}
|E_0| = 3n -6 -\sum_{i=1}^f t(\Phi_i).
\end{equation}

As $\sum_{i=1}^f c(\Phi_i)=\sum_{i=1}^f c'(\Phi_i), $ \Cref{main_lemma} yields
\begin{align*}
    \sum_{i=1}^f c(\Phi_i) &\geq 0,\\
    \sum_{i=1}^f h(\Phi_i) &\leq  \sum_{i=1}^f \left( 2 \cdot t(\Phi_i) +|\Phi_i| -2 \right).
\end{align*}

Since every edge of $G_1$ contains two halfedges, we have

\begin{equation} \label{eq_halfedges}
|E_1| \leq \frac{1}{2} \sum_{i=1}^f \left( 2 \cdot t(\Phi_i) +|\Phi_i| -2 \right).
\end{equation}

Summing inequalities \eqref{eq_edges} and \eqref{eq_halfedges}, we obtain

\begin{align*}
    |E| &\leq 3n-6  -\sum_{i=1}^f t(\Phi_i)  + \frac{1}{2} \sum_{i=1}^f \left( 2 \cdot t(\Phi_i) +|\Phi_i| -2 \right)\\
     &= 3n-6 + \frac{1}{2} \sum_{i=1}^f (|\Phi_i| -2) \le 4n-8,
\end{align*}

where in the last step, we used the fact that $\sum_{i=1}^f (|\Phi_i| -2) \le 2n-4$. This completes the proof of \Cref{thm_2planar}. 
\end{proof}

%%%%%%%%%%%%%%%%%%%%%%%%%%%%%%%%%%%%%%%%

\smallskip
We continue with the lower bound. 
\begin{proof}[\bf Proof of \Cref{thm_construction}.]

Suppose first that $n=69k^2-57k+17$ for some integer $k$. We begin by describing our construction in this special case.
The base unit is a 2-planar unit distance graph $H$ with $9$ vertices and $18$ edges, shown in \Cref{n=9}. This graph is the $3\times 3$ Rook's graph (it is also the Paley graph of order 9 and the Hamming graph $H(2,3)$).
This graph has the most edges among all unit distance graphs on $9$ vertices.

    \begin{figure}[ht!]
    \centering
        \includegraphics[scale=0.33]{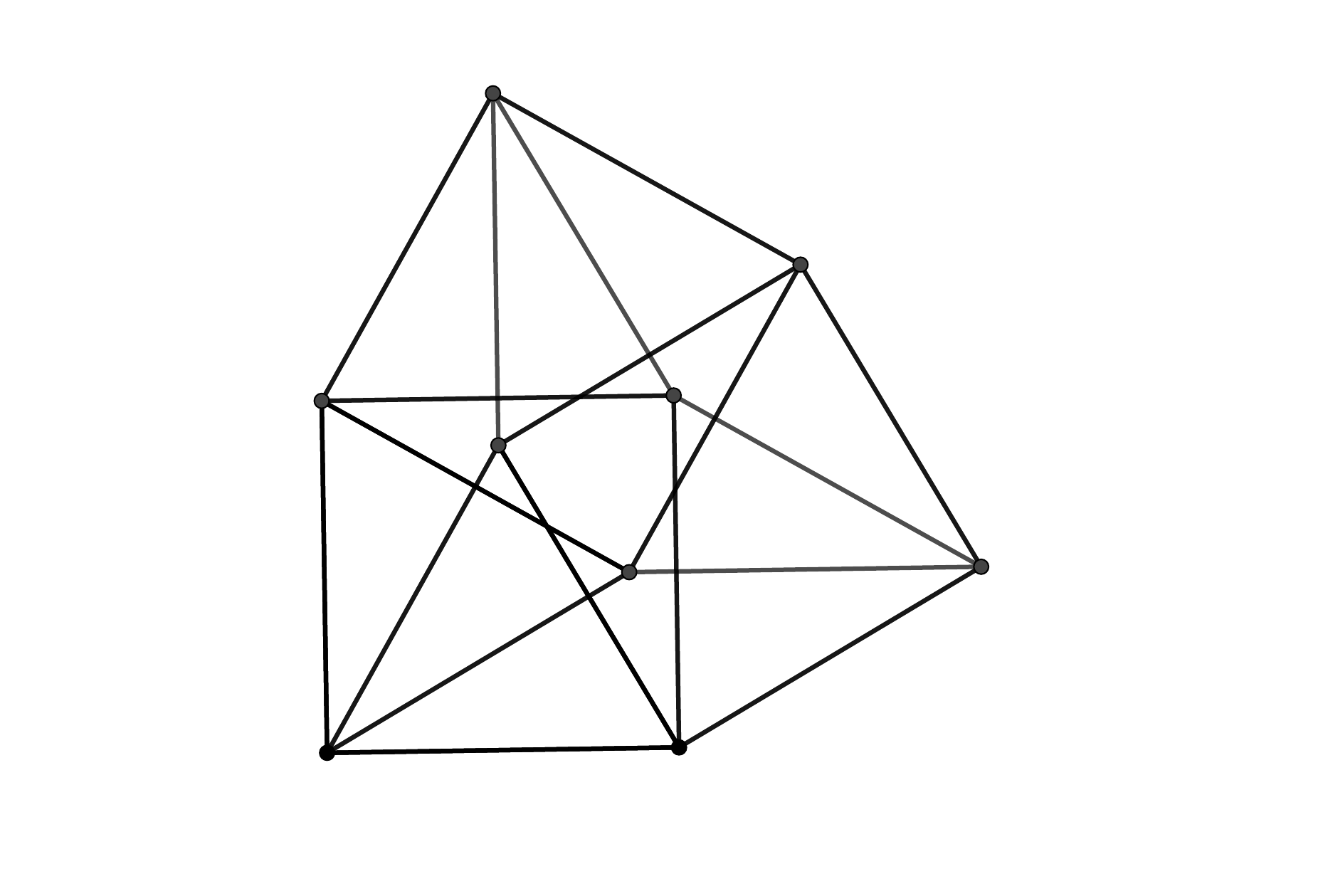}
        \caption{The basis of the construction. It has $9$ vertices and $18$ edges.}
        \label{n=9}
    \end{figure}   
    
We build our construction using $H$ as unit blocks. First, using four rotated copies of $H$ and four extra edges, we create a regular dodecagon, as shown in \Cref{n=29}, where the additional edges are shown with dashed lines. Clearly, it is a $2$-planar unit distance graph. It contains $29$ vertices and $72$ edges.

       \begin{figure}[ht!]
      \centering
        \includegraphics[scale=0.38]{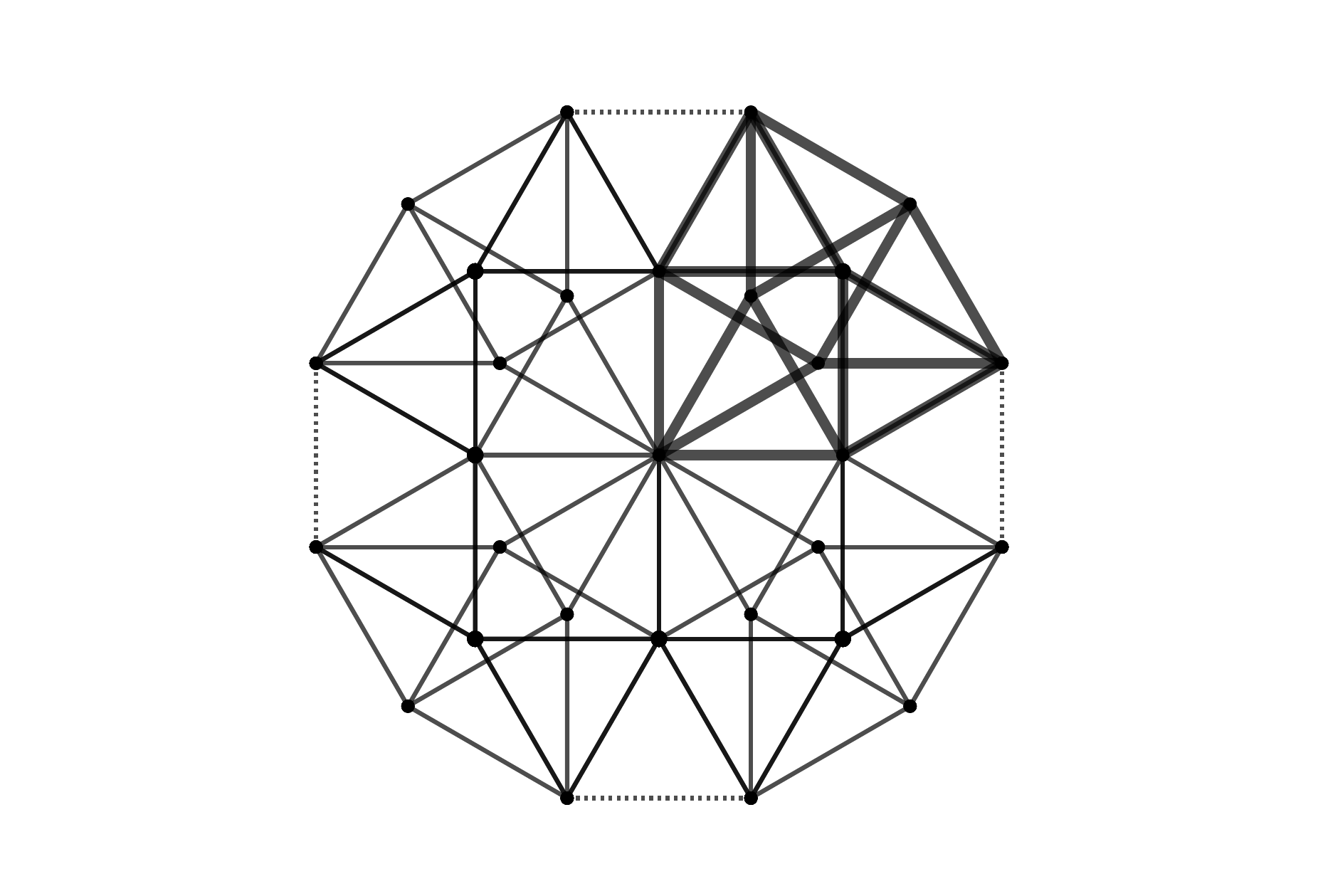}
        \caption{The regular dodecagon. We rotated the original $9$-point construction (drawn with thick edges) around its lower left vertex to create a dodecagon. We also added four extra edges (drawn with dashed lines). The graph has $29$ vertices and $72$ edges.}
        \label{n=29}
    \end{figure}
    
Next, we use these dodecagons to form a piece of a lattice. We join the dodecagons in a way such that every second side of the dodecagon is touching another one. In other words, we build a hexagonal grid using the dodecagons. We build the grid following a spiral line, starting from the center. 
Let $k$ denote the number of layers, that is, if the grid contains one dodecagon, we have $k=1$, if the grid contains $7$ dodecagons, we have $k=2$, if the grid contains $19$ dodecagons, we have $k=3$, etc. In \Cref{k=2}, we show how the dodecagons are placed in a hexagonal grid via the case $k=2$. Furthermore, between any two neighboring boundary dodecagons, we add an extra edge, colored with gray in the figure.

       \begin{figure}[ht!]
        \centering
        \includegraphics[scale=0.4]{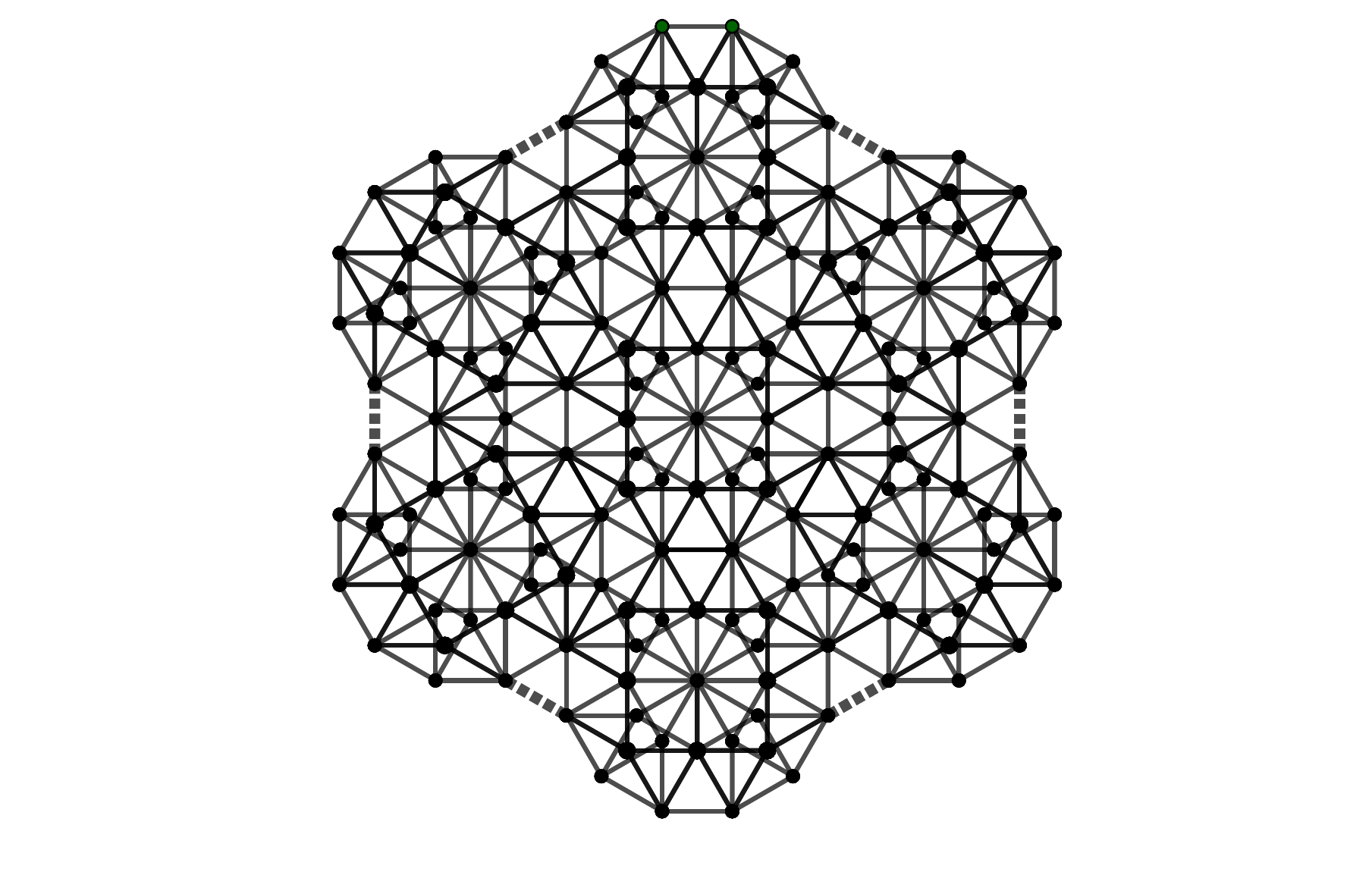}
        \caption{The $k=2$ case. We join $7$ regular dodecagons in a hexagonal grid pattern. We add one extra edge (drawn as dashed lines) between each pair of neighboring dodecagons on the boundary.}
        \label{k=2}
    \end{figure}

We denote the resulting graph containing $n$ vertices by $G_n$ and denote the number of edges in $G_n$ by $e$. Clearly, $G_n$ is a 2-planar unit distance graph.
Let us calculate the number of vertices of $G_n$ with $k$ layers. It contains $3k^2-3k+1$ dodecagons, the number of vertices in a dodecagon is $29$, but at any connection two edges coincide. Let $c_1$ denote the number of edge pairs that coincide, and let $h$ denote the number of dodecagons. For $c_1$, we have
    
$$c_1=\frac12 \left[6(3k^2-3k+1)-12k+6 \right]=3h-6k+3.$$
    
Thus, the total number of vertices needed to complete $k$ layers is $$n=29h-2c_1=23h+12k-6=69k^2-57k+17.$$ 
Next, we calculate the number of edges of $G_n$.

\begin{claim} \label{specialcase}
For any $n$ that can be written as $n=69k^2-57k+17$ where $k$ is an integer, the graph $G_n$ contains
$\left\lfloor3n-\sqrt{\frac{192}{23}n-\frac{23088}{529}}\right\rfloor$ edges.
\end{claim}

\begin{proof}
As we have already seen, as $n=69k^2-57k+17$, 
the graph $G_n$ corresponds to the construction with $k$ completed layers.
Each dodecagon contains $72$ edges, at each connection one edge is double-counted, and between any two neighboring dodecagons on the boundary of the grid we added an extra (gray) edge. Let $c_2$ denote the number of neighboring pairs on the boundary. Clearly, we have $c_2=6(k-1)$. Thus,    
$$e =72h-c_1+c_2=72(3k^2-3k+1)-3(3k^2-3k+1)+6k-3+6k-6= 207k^2-195k+60.$$

    As $n=69k^2-57k+17$, we have:
   \begin{align*}
e &=207k^2-195k+60=\left\lfloor207k^2-195k+60+\frac{21}{23}\right\rfloor=\left\lfloor207k^2-171k+51-\sqrt{\left(24k-\frac{228}{23}\right)^2}\right\rfloor
\\
&=\left\lfloor207k^2-171k+51-\sqrt{576k^2-\frac{10944}{23}k+\frac{3264}{23}-\frac{23088}{529}}\right\rfloor=\left\lfloor3n-\sqrt{\frac{192}{23}n-\frac{23088}{529}}\right\rfloor.
    \end{align*}

    Hence, the statement is proved.
\end{proof}

Now we investigate the general case. We can assume that $n$ can be written as $n=69k^2-57k+17+A$ for some integers $k$ and $A<[69(k+1)^2-57(k+1)+17]-[69k^2-57k+17]=138k+12$. 
We extend the construction in the following way. First, we consider the graph $G_m$ for $m=69k^2-57k+17$ and we use the remaining $A$ vertices as follows: we choose a position adjacent to two dodecagons of the previous layer and we begin building a dodecagon. We continue adding vertices to this dodecagon until either the dodecagon is completed, or no vertices remain. Once a dodecagon is fully formed, we proceed to the next position in clockwise order in the $(k+1)$-th layer. We only start building a new dodecagon when the previous one is completed, and repeat this procedure until we place $A$ new vertices.
We still denote the resulting graph containing $n$ vertices by $G_n$. \Cref{thm_construction} follows directly from the following statement.
\begin{theorem}
For all $n \geq 179$, the graph $G_n$ contains
$\left\lfloor3n-\sqrt{\frac{192}{23}n-\frac{23088}{529}}-\frac{13683}{23}\right\rfloor$ edges.
\end{theorem}

\begin{proof}
Let $n=69k^2-57k+17+A$ for some integers $k$ and nonnegative $A<138k+12$ where $n \geq 179$. That is, $G_n$ contains $k \geq 2$ completed layers of the hexagonal grid and an incomplete layer with $A \geq 0$ vertices. By \Cref{specialcase} the $k$ layers contain $207k^2-195k+60$ edges. Consider the remaining $A$ vertices in the $(k+1)$-th layer. When building the $(k+1)$-th layer along the spiral, the first dodecagon requires $25$ extra vertices and the next ones always require either $25$ or $23$ extra vertices, contributing $71$ or $69$ new edges, respectively.
Furthermore, there are only $7$ dodecagons requiring $25$ vertices (the first one and the ones on the corners of the hexagonal grid), which consume at most $7 \cdot 25 = 175$ vertices. For simplicity, we omit these dodecagons. Thus, there are at least $\left\lfloor\frac{A-175}{23}\right\rfloor\geq \frac{A-198}{23}$ fully formed dodecagons in the $(k+1)$-th layer, each contributing $69$ edges.

Thus, for the number of edges of $G_n$, we have
\begin{align}
e&\geq 207k^2-195k+60+\frac{A-198}{23}\cdot69  \nonumber\\
&=207k^2-195k+3A-534 \nonumber\\
&= 3n-24k-585 \label{eq9} \\
&\geq 3n - 24 \left( \frac{57 + \sqrt{276n - 1443}}{138} \right) - 585 \label{eq10}\\
&= 3n - \sqrt{\frac{192}{23}n - \frac{23088}{529}} - \frac{13683}{23}. \nonumber
\end{align}

In \eqref{eq9} we used the fact that $3n=207k^2-171k+51+3A$. In \eqref{eq10}, we used that $n=69k^2-57k+17+A$, $A \geq 0$ and therefore
$k \leq \frac{57 + \sqrt{276n - 1443}}{138}.$

This proves the statement.
\end{proof}

To finish the proof of \Cref{thm_construction}, recall that $u_0(n) = 3n-\left\lfloor\sqrt{12n - 3}\right\rfloor$. Since $\frac{192}{23} \approx 8.34 < 12$, for large enough values of $n$, the number of edges in our construction exceeds $u_0(n)$. The difference is asymptotically $\left(\sqrt{12} - \sqrt{192/23}\right)\sqrt{n} \approx 0.58\sqrt{n}$.
\end{proof}

%%%%%%%%%%%%%%%%%%%%%%%%%%%%%%%%%%%%%%%%%%%%%%%%%%%%%%%%%%%%%%%%%%%%%
\bigskip
\section*{Acknowledgments}

The authors are very grateful to Lili K\"{o}dm\"{o}n for helpful discussions.

%%%%%%%%%%%%%%%%%%%%%%%%%%%%%%%%%%%%%%%%%%%%%%%%%%%%


\newpage
\small
\begin{thebibliography}{30} 


\bibitem{A19}
Ackerman, E.\ (2019).
{\it On topological graphs with at most four crossings per edge.}
Computational Geometry, 85, 101574.

\bibitem{BSW83}
Von Bodendiek, R., Schumacher, H., Wagner, K.\ (1983).
{\it Bemerkungen zu einem Sechsfarbenproblem von G.\ Ringel.}
In Abhandlungen aus dem Mathematischen Seminar der Universität Hamburg (Vol.\ 53, No. 1, pp.\ 41--52). Berlin/Heidelberg: Springer-Verlag.

\bibitem{C13}
do Carmo, M.\ P.\ (2013).
{\it Differentialgeometrie von Kurven und Fl\"achen.}
Vol.\ 55, Springer-Verlag.


\bibitem{CK25}
\v{C}ervenkov\'a, E., Kratochv\'il, J.\ (2025).
{\it 1-Planar Unit Distance Graphs with More Edges Than Matchstick Graphs.}
In 33rd International Symposium on Graph Drawing and Network Visualization (GD 2025) (pp.\ 26--1). Schloss Dagstuhl-Leibniz-Zentrum für Informatik.


\bibitem{GT25}
Geh\'er, P., Tóth, G.\ (2025).
{\it 1-planar unit distance graphs.}
European Journal of Combinatorics, 130, 104212.

\bibitem{H74}
Harborth, H.\ (1974).
{\it Solution to problem 664A.}
Elemente der Mathematik 29, 14--15.

\bibitem{H81}
Harborth, H.\ (1981).
{\it Point sets with equal numbers of unit-distant neighbors.}
Discrete Geometry, 12--18.

\bibitem{H94}
Harborth, H.\ (1994).
{\it Match sticks in the plane.}
In The lighter side of mathematics. Proceedings of the Eugne Strens memorial conference on recreational mathematics and its history (pp.\ 281--288).

\bibitem{KKKRSU23}
Kaufmann, M., Klemz, B., Knorr, K., Reddy, M.\ M., Schröder, F., Ueckerdt, T.\ (2024).
{\it The Density Formula: One Lemma to Bound Them All.} In 32nd International Symposium on Graph Drawing and Network Visualization (GD 2024) (Vol.\ 320, pp.\ 7--1). Schloss Dagstuhl–Leibniz-Zentrum für Informatik.

\bibitem{LS22}
Lavoll\'ee, J., Swanepoel, K.\ (2022).
{\it A tight bound for the number of edges of matchstick graphs.}
Discrete and Computational Geometry, 1--15.


\bibitem{PT97}
Pach, J., T\'oth, G.\ (1997).
{\it Graphs drawn with few crossings per edge.}
Combinatorica, 17(3), 427--439.

\end{thebibliography}
\end{document}